\documentclass[a4paper]{amsart}

\usepackage{amsmath,amssymb,amsthm}
\usepackage{eucal,here}
\usepackage[all]{xy}
\usepackage{tikz}
\usepackage[dvipdfmx,colorlinks=true, backref=page]{hyperref}

\newtheorem{theorem}{Theorem}[section]
\newtheorem{proposition}[theorem]{Proposition}
\newtheorem{lemma}[theorem]{Lemma}
\newtheorem{corollary}[theorem]{Corollary}

\theoremstyle{definition}

\newtheorem{problem}[theorem]{Problem}

\theoremstyle{remark}

\numberwithin{equation}{section}

\newcommand{\Q}{\mathbb{Q}}

\newcommand{\C}{\mathcal{C}}
\newcommand{\sgn}{\operatorname{sgn}}

\SelectTips{cm}{}

\title[Tight complexes are Golod]{Tight complexes are Golod}

\author{Kouyemon Iriye}
\address{Department of Mathematics, Osaka Metropolitan University, Sakai, 599-8531, Japan}
\email{kiriye@omu.ac.jp}

\author{Daisuke Kishimoto}
\address{Faculty of Mathematics, Kyushu University, Fukuoka 819-0395, Japan}
\email{kishimoto@math.kyushu-u.ac.jp}

\date{\today}



\begin{document}

\maketitle

\begin{abstract}
  Golodness of a simplicial complex is defined algebraically in terms of the Stanley-Reisner ring, and it has been a long-standing problem to find its combinatorial characterization. Tightness of a simplicial complex is a combinatorial analogue of a tight embedding of a manifold into Euclidean space, and has been studied in connection to minimal manifold triangulations. In this paper, by using an idea from toric topology, we prove that tight complexes are Golod, and as a corollary, we obtain that a triangulation of a closed connected orientable manifold is Golod if and only if it is tight, which is a combinatorial characterization of Golodness for manifold triangulations.
\end{abstract}


\section{Introduction}\label{introduction}

Let $K$ be a simplicial complex, and let $\mathbb{F}[K]$ denote the Stanley-Reisner ring of $K$ over a field $\mathbb{F}$. The following problem is of particular interest.

\begin{problem}
  \label{problem}
  Given an algebraic property of $\mathbb{F}[K]$, find a combinatorial characterization of it.
\end{problem}

There are some cases where Problem \ref{problem} is resolved. A typical example is Cohen-Macaulayness which is characterized by Reisner's criterion. In this paper, we consider Golodness in Problem \ref{problem}. Let $\mathbb{F}$ be a field, and let $S=\mathbb{F}[x_1,\ldots,x_n]$ where we set each $x_i$ to be of degree 2. Serre \cite{S} proved that for a noetherian ring $R=S/I$ where $I$ is a homogeneous ideal of $S$ such that $I\subset(x_1,\ldots,x_n)^2$, there is a coefficient-wise inequality
\[
  P(\mathrm{Tor}^R(\mathbb{F},\mathbb{F});t)\le\frac{(1+t^2)^n}{1-t(P(\mathrm{Tor}^S(R,\mathbb{F});t)-1)}
\]
where $P(V;t)$ denotes the Poincar\'e series of a graded vector space $V$ in the indeterminate $t$. If equality holds, then $R$ is called \emph{Golod}. Later, it was Golod who proved that $R$ is Golod if and only if all products and (higher) Massey products in the Koszul homology of $R$ (with respect to $x_1,\ldots,x_n$) are trivial. Therefore Golodness is a fairly complex property.

 We say that $K$ is \emph{$\mathbb{F}$-Golod} if $\mathbb{F}[K]$ is Golod. As Golodness is quite complicated as mentioned above, Problem \ref{problem} for Golodness has been widely open for a long time. Recently, toric topology provided new techniques and viewpoints for attacking Problem \ref{problem} for Golodness. More precisely, as in \cite{BBP}, a space called the moment-angle complex for $K$, or a more general polyhedral product, realizes the Koszul complex of $\mathbb{F}[K]$ as its cellular cochain complex, up to homotopy, and so topological techniques apply to produce new results on Golodness of a simplicial complex as in \cite{BG,GPTW,GT,GW,IK0,IK1,IK2,IK3,IK4,IY,L}. Using this connection of topology to combinatorics, other results in combinatorics and topology were obtained as in \cite{AP,HKK,IK5,IK6,IKL,KMY,Ki,KL}. See the survey \cite{BBC} and references therein for more details and results. Remarkably, we can get a new point of view that $\mathbb{F}$-Golod manifold triangulations might be connected to the minimality of triangulations as follows. As in \cite{GPTW,IK4}, a triangulation of a circle is $\mathbb{F}$-Golod if and only if it is a triangle which is the minimal triangulation of a circle, and as in \cite{IK1}, a triangulation of a connected closed $\mathbb{F}$-orientable surface is $\mathbb{F}$-Golod if and only if it is neighborly, i.e. every two vertices are joined by an edge, where most neighborly triangulations are minimal as in \cite{Lu}.

We give the definition of a tight complex. For  a non-empty subset $I$ of the vertex set of a simplicial complex $K$, the full subcomplex of $K$ over $I$ is defined by
\[
  K_I=\{\sigma\in K\mid\sigma\subset I\}.
\]
A simplicial complex $K$ is called \emph{$\mathbb{F}$-tight} if the natural map
\[
  H_*(K_I;\mathbb{F})\to H_*(K;\mathbb{F})
\]
is injective for each non-empty subset $I$ of the vertex set of $K$. Tightness of a simplicial complex was introduced by K\"{u}hnel \cite{Ku} as a combinatorial analogue of a tight embedding of a manifold into the Euclidean space, and since then, it has been intensely studied. In particular, K\"{u}hnel and Lutz \cite{KL} conjectured that tight manifold triangulations are strongly minimal, that is, the number of simplices are minimal in each dimension, and this conjecture was verified in dimension $\le 3$ \cite{BDS,Ku,KL}. Then since we may expect a connection of Golodness of manifold triangulations and the minimality of triangulations, we may also expect a connection between Golodness and tightness for manifold triangulations, which is of particular interest because the two notions from different contexts might be equivalent. By using the moment-angle complex, the authors \cite{IK5} proved topologically that if $K$ is a triangulation of a closed connected $\mathbb{F}$-orientable 3-manifold, then $K$ is $\mathbb{F}$-Golod if and only if it is $\mathbb{F}$-tight. Moreover, the authors \cite{IK5} also proved that if a simplicial complex, not necessarily a manifold triangulation, is $\mathbb{F}$-tight, then it is weakly $\mathbb{F}$-Golod, that is, all products in the Koszul homology of $\mathbb{F}[K]$ are trivial.

In this paper, we will prove:

\begin{theorem}
  \label{main}
  Every $\mathbb{F}$-tight complex is $\mathbb{F}$-Golod.
\end{theorem}

The converse of Theorem \ref{main} does not hold, in general. For example, let $K$ be a simplicial complex consisting of two vertices. Then the Koszul homology of $\mathbb{F}[K]$ is the exterior algebra generated by a single element of degree 3, and so $K$ is $\mathbb{F}$-Golod. But $K$ is not $\mathbb{F}$-tight because $\mathbb{F}$-tight complexes must be neighborly.

Now we restrict simplicial complexes, and consider the converse of Theorem \ref{main}. We say that a simplicial complex $K$ is \emph{weakly $\mathbb{F}$-Golod} if all products, not necessarily all (higher) Massey products, in the Koszul homology of $\mathbb{F}[K]$ are trivial. Clearly, if a simplicial complex is $\mathbb{F}$-Golod, then it is weakly $\mathbb{F}$-Golod. It is proved in \cite{IK5} that if a triangulation of a closed connected $\mathbb{F}$-orientable manifold is weakly $\mathbb{F}$-Golod, then it is $\mathbb{F}$-tight. So by Theorem \ref{main}, we get:

\begin{corollary}
  \label{tight=Golod}
  A triangulation of a closed connected $\mathbb{F}$-orientable manifold is $\mathbb{F}$-Golod if and only if it is $\mathbb{F}$-tight.
\end{corollary}

By definition, it is hard to directly check Golodness. But by Corollary \ref{tight=Golod}, we get a criterion for Golodness, which is easy to check. Let $K$ be a triangulation of a connected manifold of dimension $d\ge 3$. We say that $K$ is \emph{tight-neighborly} if
\[
  \binom{m-d-1}{2}=\binom{d+2}{2}\dim H_1(K;\Q)
\]
where $m$ is the number of vertices of $K$. It is easy to check tight-neighborliness. As in \cite{BDS,DM}, if $K$ is tight-neighborly, then it is $\mathbb{F}$-tight for any field $\mathbb{F}$, where the converse holds in dimension $3$. Thus by Corollary \ref{tight=Golod}, if $K$ is tight-neighborly, then it is $\mathbb{F}$-Golod for any field $\mathbb{F}$.

It is quite hard to show whether or not a given weakly $\mathbb{F}$-Golod is $\mathbb{F}$-Golod because we need to check the triviality of all (higher) Massey products. But there are classes of simplicial complexes where $\mathbb{F}$-Golodness and weak $\mathbb{F}$-Golodness are equivalent \cite{IK1,IK2,IK4,IK5,Ka}. (Berglund and J\"{o}llenbeck \cite{BJ} claimed that it is always the case for every simplicial complex, but Katth\"{a}n \cite{Ka} proved that this claim is false.) By the result on weak Golodness and tightness for manifold triangulations mentioned above \cite{IK5} together with Theorem \ref{main}, we can find an interesting such class.

 \begin{corollary}
   A triangulation of a closed connected $\mathbb{F}$-orientable manifold is $\mathbb{F}$-Golod if and only if it is weakly $\mathbb{F}$-Golod.
 \end{corollary}

We briefly explain the outline of the proof of Theorem \ref{main}. In Section \ref{the complex}, we will introduce the differential graded algebra (dga, for short) $\mathcal{C}^*(K)$ which is homotopy equivalent to the Koszul complex of $\mathbb{F}[K]$, whereas the dga $\mathcal{C}^*(K)$ is isomorphic with the cellular cochain complex of the moment-angle complex for $K$. Then we need to prove that all products and (higher) Massey products in $H^*(\mathcal{C}^*(K))$ are trivial whenever $K$ is $\mathbb{F}$-tight. Every possibly non-trivial product of $\mathcal{C}^*(K)$ is given by the map
\[
  K_{I\sqcup J}\to K_I\star K_J,\quad\sigma\mapsto(\sigma\cap I)\sqcup(\sigma\cap J)
\]
in cohomology for disjoint non-empty subsets $I,J$ of the vertex set of $K$, where $K_1\star K_2$ denotes the join of simplicial complexes $K_1$ and $K_2$. The composite of it with the inclusion $K_I\star K_J\to K\star K$ is null-homotopic because it is homotopic to the inclusion into the first factor $K\subset K\star K$, which is null-homotopic. If $K$ is $\mathbb{F}$-tight, then the triviality of the composite $K_{I\cup J}\to K\star K$ in cohomology implies the triviality of the map $K_{I\sqcup J}\to K_I\star K_J$ in cohomology, implying the triviality of all products in $H^*(\mathcal{C}^*(K))$. We will generalize this argument to (higher) Massey products by considering higher homotopies. Specifically, the map
\[
  K_{I_0\sqcup\cdots\sqcup I_q}\to K^{\star(q+1)},\quad\sigma\mapsto(\sigma\cap I_0)\sqcup\cdots\sqcup(\sigma\cap I_q)
\]
for pairwise disjoint non-empty subsets $I_0,\ldots,I_q$ of the vertex set of $K$ can be deformed into $K^{\star(q+1)}\supset K^{\star q}\supset\cdots\supset K^{\star 2}\supset K$ in order by a higher homotopy that we will construct in Section \ref{higher prism operator}. If $K$ is $\mathbb{F}$-tight, then we will see that the triviality of $(q+1)$-fold Massey products can be deduced by considering the higher prism operator corresponding to the above higher homotopy, which will be constructed in Section \ref{higher prism operator} and applied in Section \ref{proof of theorem}.


\subsection*{Acknowledgements}

The authors were supported by JSPS KAKENHI Grant Numbers JP19K03473 and JP17K05248.


\section{Higher prism operator}\label{higher prism operator}

In this section, we define higher prism operators, by which we will control (higher) Massey products. Hereafter, algebras and modules will be considered over a field $\mathbb{F}$. We set notation that we are going to use throughout the paper. Let $K$ and $L$ be simplicial complexes. We define the join of $K$ and $L$ by
\[
  K\star L=\{\sigma\sqcup\tau\mid\sigma\in K,\,\tau\in L\}.
\]
Let $K^{\star n}$ denote the join of $n$ copies of $K$. Let $C^*(K)$ denote the cochain complex of a simplicial complex $K$. Then for $a\in C^p(K)$ and $b\in C^q(L)$, we can define $a\star b\in C^{p+q+1}(K\star L)$ by $(a\star b)(\sigma\sqcup\tau)=a(\sigma)b(\tau)$ for $\sigma\in K$ and $\tau\in L$, and so we get a homomorphism
\[
  \star\colon C^p(K)\otimes C^q(L)\to C^{p+q+1}(K\star L),\quad a\otimes b\mapsto a\star b.
\]
Let $V(K)$ denote the vertex set of a simplicial complex $K$. For $\emptyset\ne I\subset V(K)$, let
\[
  j_I\colon K_I\to K
\]
denote the inclusion. For $\emptyset\ne I,J\subset V(K)$ with $I\cap J=\emptyset$, we define an injective simplicial map
\[
  \iota_{I,J}\colon K_{I\cup J}\to K_I\star K_J,\quad\sigma\mapsto(\sigma\cap I)\sqcup(\sigma\cap J).
\]

To define a prism operator, we need shuffles. Recall that a $(p,q)$-shuffle is a sequence $(s_0,s_1,\ldots,s_{q+1})$ satisfying $0=s_0\le s_1\le\cdots\le s_q\le s_{q+1}=p$. Let $\mathcal{S}(p,q)$ denote the set of all $(p,q)$-shuffles. It is useful to think of a $(p,q)$-shuffle $(s_0,s_1,\ldots,s_{q+1})$ as the following lattice path in $\{0,1,\ldots,p\}\times\{0,1,\ldots,q\}$.

\begin{figure}[H]
  \centering
\begin{tikzpicture}[x=0.8cm, y=0.8cm, thick]
  \draw(0,-4)--(1.5,-4);
  \draw[dashed](1.5,-4)--(2.5,-4);
  \draw(2.5,-4)--(3,-4)--(3,-3)--(4.5,-3);
  \draw[dashed](4.5,-3)--(5.5,-3);
  \draw(5.5,-3)--(6,-3)--(6,-2)--(6.5,-2);
  \draw[dashed](6.5,-2)--(7.5,-1);
  \draw(7.5,-1)--(8,-1)--(8,0)--(9.5,0);
  \draw[dashed](9.5,0)--(10.5,0);
  \draw(10.5,0)--(11,0);
  \draw[fill=black](0,-4)circle(2pt);
  \draw(-0.9,-4.5)node{\footnotesize$(0,0)=(s_0,0)$};
  \draw[fill=black](1,-4)circle(2pt);
  \draw(1.3,-4.5)node{\footnotesize$(s_0+1,0)$};
  \draw[fill=black](3,-4)circle(2pt);
  \draw(3,-4.5)node{\footnotesize$(s_1,0)$};
  \draw[fill=black](3,-3)circle(2pt);
  \draw(3,-2.5)node{\scriptsize$(s_1,1)$};
  \draw[fill=black](4,-3)circle(2pt);
  \draw(4,-3.5)node{\footnotesize$(s_1+1,1)$};
  \draw[fill=black](6,-3)circle(2pt);
  \draw(6,-3.5)node{\footnotesize$(s_2,1)$};
  \draw[fill=black](6,-2)circle(2pt);
  \draw(6,-1.5)node{\footnotesize$(s_2,2)$};
  \draw[fill=black](8,-1)circle(2pt);
  \draw(8.3,-1.5)node{\footnotesize$(s_q,q-1)$};
  \draw[fill=black](8,0)circle(2pt);
  \draw(7.7,0.5)node{\footnotesize$(s_q,q)$};
  \draw[fill=black](9,0)circle(2pt);
  \draw(9.2,0.5)node{\footnotesize$(s_q+1,q)$};
  \draw[fill=black](11,0)circle(2pt);
  \draw(11.55,0.5)node{\footnotesize$(s_{q+1},q)=(p,q)$};
\end{tikzpicture}
\caption{}
\label{lattice path picture}
\end{figure}

\noindent For example, the latice path corresponding to the $(6,3)$-shuffle $(0,2,3,3,6)$ is given by the following figure.

\begin{figure}[htbp]
  \centering
\begin{tikzpicture}[x=0.8cm, y=0.8cm, thick]
  \draw(0,-3)--(2,-3)--(2,-2)--(3,-2)--(3,0)--(6,0);
  \draw[fill=black](0,-3)circle(2pt);
  \draw(0,-3.5)node{$(0,0)$};
  \draw[fill=black](1,-3)circle(2pt);
  \draw(1,-3.5)node{$(1,0)$};
  \draw[fill=black](2,-3)circle(2pt);
  \draw(2,-3.5)node{$(2,0)$};
  \draw[fill=black](2,-2)circle(2pt);
  \draw(2,-1.5)node{$(2,1)$};
  \draw[fill=black](3,-2)circle(2pt);
  \draw(3,-2.5)node{$(3,1)$};
  \draw[fill=black](3,-1)circle(2pt);
  \draw(3.75,-1)node{$(3,2)$};
  \draw[fill=black](3,0)circle(2pt);
  \draw(3,0.55)node{$(3,3)$};
  \draw[fill=black](4,0)circle(2pt);
  \draw(4,0.55)node{$(4,3)$};
  \draw[fill=black](5,0)circle(2pt);
  \draw(5,0.55)node{$(5,3)$};
  \draw[fill=black](6,0)circle(2pt);
  \draw(6,0.55)node{$(6,3)$};
\end{tikzpicture}
\caption{}
\label{lattice path picture special}
\end{figure}

\noindent Let $s=(s_0,\ldots,s_{q+1})\in\mathcal{S}(p,q)$. We define
\begin{align*}
  s(\rightarrow)
  &=\{0<i<p+q\mid s_k+k<i<s_{k+1}+k\text{ for some }0\le k\le q\}\cup A\cup B\\
  s(\uparrow)
  &=\{0\le i<p+q\mid i=s_k+k\text{ and }s_k=s_{k+1}\text{ for some }0\le k\le q\}\cup C\\
  s(\rotatebox[origin=c]{-90}{$\Lsh$})
  &=\{0<i<p+q\mid i=s_k+k-1\text{ and }s_{k-1}<s_k\text{ for some }0\le k\le q+1\}\\
  s(\Rsh)
  &=\{0<i<p+q\mid i=s_k+k\text{ and }s_k<s_{k+1}\text{ for some }0\le k\le q\}
\end{align*}
where we set
\[
  A=
  \begin{cases}
    \{0\}&s_0<s_1\\
    \emptyset&s_0=s_1
  \end{cases}
  \quad B=
  \begin{cases}
    \{p+q\}&s_q<s_{q+1}\\
    \emptyset&s_q=s_{q+1}
  \end{cases}
  \quad C=
  \begin{cases}
    \emptyset&s_q<s_{q+1}\\
    \{p+q\}&s_q=s_{q+1}.
  \end{cases}
\]
Then each $0\le i\le p+q$ belongs to one of these sets, and for $0<i<p+q$, elements of these sets are the $i$-th vertex of the lattice path of Figure \ref{lattice path picture} depicted as the black vertex in the following figure.

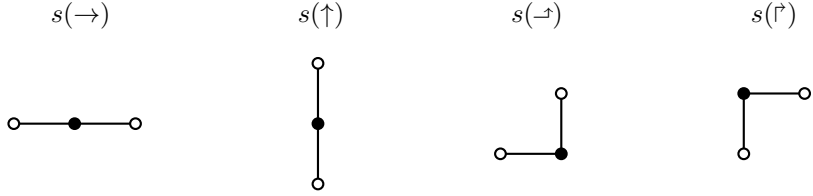
\begin{figure}[H]
  \centering
\begin{tikzpicture}[x=0.8cm, y=0.8cm, thick]
  \draw(0.5,0)--(2.5,0);
  \draw[fill=white](0.5,0)circle(2pt);
  \draw[fill=black](1.5,0)circle(2pt);
  \draw[fill=white](2.5,0)circle(2pt);
  \draw(1.6,1.8)node{$s(\rightarrow)$};
  %
  \draw(5.5,1)--(5.5,-1);
  \draw[fill=white](5.5,1)circle(2pt);
  \draw[fill=black](5.5,0)circle(2pt);
  \draw[fill=white](5.5,-1)circle(2pt);
  \draw(5.55,1.8)node{$s(\uparrow)$};
  %
  \draw(9.5,0.5)--(9.5,-0.5)--(8.5,-0.5);
  \draw[fill=white](9.5,0.5)circle(2pt);
  \draw[fill=black](9.5,-0.5)circle(2pt);
  \draw[fill=white](8.5,-0.5)circle(2pt);
  \draw(9.1,1.8)node{$s(\rotatebox[origin=c]{-90}{$\Lsh$})$};
  %
  \draw(13.5,0.5)--(12.5,0.5)--(12.5,-0.5);
  \draw[fill=white](13.5,0.5)circle(2pt);
  \draw[fill=black](12.5,0.5)circle(2pt);
  \draw[fill=white](12.5,-0.5)circle(2pt);
  \draw(13,1.8)node{$s(\Rsh)$};
\end{tikzpicture}
\end{figure}

\noindent Moreover, $0$ and $p+q$ belong to either $s(\rightarrow)$ or $s(\uparrow)$ if they are the black vertices in the following figure.

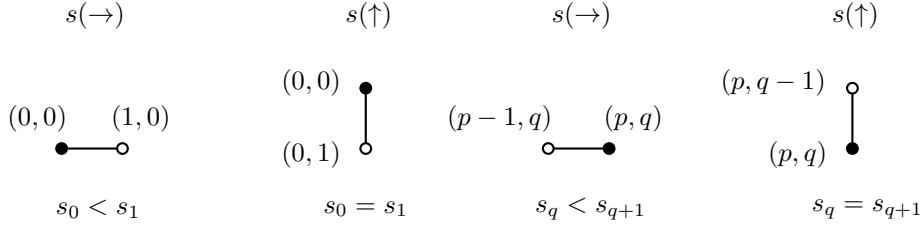
\begin{figure}[H]
  \centering
\begin{tikzpicture}[x=0.8cm, y=0.8cm, thick]
  \draw(0,0)--(1,0);
  \draw[fill=black](0,0)circle(2pt);
  \draw(-0.4,0.5)node{$(0,0)$};
  \draw[fill=white](1,0)circle(2pt);
  \draw(1.3,0.5)node{$(1,0)$};
  \draw(0.55,2.2)node{$s(\rightarrow)$};
  \draw(0.6,-1)node{$s_0<s_1$};
  \draw(5,1)--(5,0);
  \draw[fill=black](5,1)circle(2pt);
  \draw(4.1,1.1)node{$(0,0)$};
  \draw[fill=white](5,0)circle(2pt);
  \draw(4.1,-0.1)node{$(0,1)$};
  \draw(5.05,2.2)node{$s(\uparrow)$};
  \draw(5,-1)node{$s_0=s_1$};
  \draw(8,0)--(9,0);
  \draw[fill=white](8,0)circle(2pt);
  \draw(7.2,0.5)node{$(p-1,q)$};
  \draw[fill=black](9,0)circle(2pt);
  \draw(9.4,0.5)node{$(p,q)$};
  \draw(8.55,2.2)node{$s(\rightarrow)$};
  \draw(8.7,-1)node{$s_q<s_{q+1}$};
  \draw(13,1)--(13,0);
  \draw[fill=white](13,1)circle(2pt);
  \draw(11.7,1.1)node{$(p,q-1)$};
  \draw[fill=black](13,0)circle(2pt);
  \draw(12.1,-0.1)node{$(p,q)$};
  \draw(13.05,2.2)node{$s(\uparrow)$};
  \draw(13.25,-1)node{$s_q=s_{q+1}$};
\end{tikzpicture}
\end{figure}

\noindent Then we have
\[
  \{0,1,\ldots,p+q\}=s(\rightarrow)\sqcup s(\uparrow)\sqcup s(\rotatebox[origin=c]{-90}{$\Lsh$})\sqcup s(\Rsh)
\]
where this partition characterizes the shuffle $s$. For example, for a $(6,3)$-shuffle $s=(0,2,3,3,6)$, we have
\[
  s(\rightarrow)=\{0,1,7,8,9\},\quad s(\uparrow)=\{5\},\quad s(\rotatebox[origin=c]{-90}{$\Lsh$})=\{2,4\},\quad s(\Rsh)=\{3,6\}.
\]
where the corresponding lattice path is as in Figure \ref{lattice path picture special}. For $i\in s(\rotatebox[origin=c]{-90}{$\Lsh$})$ with $i=s_k+k-1$, we define a $(p,q)$-shuffle
\[
  \alpha_i(s)=(s_0,\ldots,s_{k-1},s_k-1,s_{k+1},\ldots,s_{q+1}).
\]
For $j\in s(\Rsh)$ with $j=s_l+l$, we also define a $(p,q)$-shuffle
\[
  \beta_j(s)=(s_0,\ldots,s_{l-1},s_l+1,s_{l+1},\ldots,s_{q+1}).
\]
Then the lattice path corresponding to $\alpha_i(s)$ (resp. $\beta_j(s)$) is obtained by flipping the corner at the $i$-th (resp. $j$-th) vertex in the lattice path corresponding to the $(p,q)$-shuffle $s$.

\begin{lemma}
  \label{flip}
  The map
  \[
    \coprod_{s\in\mathcal{S}(p,q)}s(\rotatebox[origin=c]{-90}{$\Lsh$})\times s\to\coprod_{s\in\mathcal{S}(p,q)}s(\Rsh)\times s,\quad(i,s)\mapsto(i,\alpha_i(s))
  \]
  is bijective.
\end{lemma}

\begin{proof}
  By definition, for $i\in s(\rotatebox[origin=c]{-90}{$\Lsh$})$ and $j\in s(\Rsh)$, we have $i\in \alpha_i(s)(\Rsh)$ and $j\in \beta_j(s)(\rotatebox[origin=c]{-90}{$\Lsh$})$. Moreover, we have
  \[
    \beta_i(\alpha_i(s))=s\quad\text{and}\quad\alpha_j(\beta_j(s))=s.
  \]
  Then the map
  \[
    \coprod_{s\in\mathcal{S}(p,q)}s(\Rsh)\times s\to\coprod_{s\in\mathcal{S}(p,q)}s(\rotatebox[origin=c]{-90}{$\Lsh$})\times s,\quad(i,s)\mapsto(i,\beta_i(s))
  \]
  is the inverse of the map in the statement.
\end{proof}

For $k=0,1,\ldots,q$, we define maps
\[
  \lambda_k\colon\mathcal{S}(p-1,q)\to\mathcal{S}(p,q),\quad(s_0,\ldots,s_{q+1})\mapsto(s_0,\ldots,s_k,s_{k+1}+1,\ldots,s_{q+1}+1)
\]
and
\[
  \nu_k\colon\mathcal{S}(p,q-1)\to\mathcal{S}(p,q),\quad(s_0,\ldots,s_q)\mapsto(s_0,\ldots,s_{k-1},s_k,s_k,s_{k+1},\ldots,s_q).
\]
Clearly, the maps $\lambda_k$ and $\nu_k$ are injective. For $s=(s_0,\ldots,s_q)\in\mathcal{S}(p,q)$, we set
\[
  \sgn(s)=(-1)^{s_1+\cdots+s_q}.
\]
Then for $s=(s_0,\ldots,s_{q+1})\in\mathcal{S}(p-1,q)$ and $t=(t_0,\ldots,t_q)\in\mathcal{S}(p,q-1)$, we have
\begin{equation}
  \label{sign}
  \sgn(\lambda_k(s))=(-1)^{q-k}\sgn(s)\quad\text{and}\quad\sgn(\nu_k(t))=(-1)^{t_k}\sgn(t).
\end{equation}

We say that a simplicial complex $K$ is ordered if $V(K)$ is equipped with a total order. Note that every simplicial complex can be an ordered simplicial complex, where any ordering will do for our purpose. Let $K$ and $L$ be ordered simplicial complexes. Then we equip $V(K)\times V(L)$ with the product order such that $(v_1,w_1)\le(v_2,w_2)\in V(K)\times V(L)$ if $v_1\le v_2$ and $w_1\le w_2$. We define the product of $K$ and $L$, denoted by $K\otimes L$, to be the simplicial complex with vertex set $V(K)\times V(L)$, whose simplices are chains $(v_0,w_0)<(v_1,w_1)<\cdots<(v_p,w_p)$ such that $\{v_0,\ldots,v_p\}\in K$ and $\{w_0,\ldots,w_p\}\in L$.

We will always assume that a $q$-simplex $\Delta^q$ is ordered in such a way that its vertex set is $\{0,1,\ldots,q\}$ with the usual ordering. Let $K$ be an ordered simplicial complex. For a $(p,q)$-shuffle $s=(s_0,\ldots,s_{q+1})$ and a $p$-simplex $\sigma=\{v_0<\cdots<v_p\}$ of $K$, we define a $(p+q)$-simplex of $K\otimes\Delta^q$ by
\[
  \sigma\odot s=[(v_{s_0},0),\ldots,(v_{s_1},0),(v_{s_1},1),\ldots,(v_{s_2},1),\ldots,(v_{s_q},q),\ldots,(v_{s_{q+1}},q)]
\]
which is quite similar to thinking of shuffles as lattice paths of Figure \ref{lattice path picture}. Let $L$ be a simplicial complex. Now, for a simplicial map $H\colon K\otimes\Delta^q\to L$, we define a higher prism operator $P_H\colon C_*(K)\to C_{*+q}(L)$ by
\[
  P_H(\sigma)=\sum_{s\in\mathcal{S}(p,q)}\sgn(s)H_*(\sigma\odot s).
\]
Observe that for $q=0$, $P_H=H_*$ and for $q=1$, $P_H$ is the usual prism operator. We prove a coherent property among higher prism operators.

\begin{lemma}
  \label{boundary condition}
  For a simplicial map $H\colon K\otimes\Delta^q\to L$, there is an identity
  \[
    \partial P_H+(-1)^{q-1}P_H\partial=\sum_{i=0}^q(-1)^iP_{H\circ(1\otimes d^i)}
  \]
  where $d^i\colon\Delta^{q-1}\to\Delta^q$ denotes the $i$-th coface operator.
\end{lemma}

\begin{proof}
  Let $s=(s_0,\ldots,s_{q+1})$ be a $(p,q)$-shuffle. Then
  \begin{align*}
    &\partial P_H(\sigma)\\
    &=\sum_{j=0}^{p+q}(-1)^j\partial_j\left(\sum_{s\in\mathcal{S}(p,q)}\sgn(s)H_*(\sigma\odot s)\right)\\
    &=\sum_{j=0}^{p+q}(-1)^j\sum_{s\in\mathcal{S}(p,q)}\sgn(s)H_*(\partial_j(\sigma\odot s))\\
    &=\sum_{s\in\mathcal{S}(p,q)}\left(\sum_{j\in s(\rightarrow)}(-1)^j\sgn(s)H_*(\partial_j(\sigma\odot s))\right.+\sum_{j\in s(\uparrow)}(-1)^j\sgn(s)H_*(\partial_j(\sigma\odot s))\\
    &\quad\left.+\sum_{j\in s(\Rsh)}(-1)^j\sgn(s)H_*(\partial_j(\sigma\odot s))+\sum_{j\in s(\rotatebox[origin=c]{-90}{$\Lsh$})}(-1)^j\sgn(s)H_*(\partial_j(\sigma\odot s))\right).
  \end{align*}
  Let $j\in s(\rightarrow)$. If $0<j<p+q$, then $s_{k_j}+k_j<j<s_{k_j+1}+k_j$ for some $0\le k_j\le q$. For $j=0,p+q$, we put $k_j=0,q$, respectively. Then we have
  \begin{align*}
    \partial_j(\sigma\odot s)&=[(v_0,0),\dots,(v_{j-k_j-1},k_j),(v_{j-k_j+1},k_j),\dots,(v_p,q)]\\
    &=(\partial_{j-k_j}\sigma)\odot\lambda_{k_j}^{-1}(s).
  \end{align*}
  Let $j\in s(\uparrow)$. If $0< j<p+q$, then $j=s_{k_j}+k_j$ and $s_{k_j}=s_{k_j+1}$ for some $0\le k_j\le q$. For $j=0,p+q$, we put $k_j=0,q$, respectively. Then we have
  \begin{align*}
    \partial_j(\sigma\odot s)&=[(v_0,0),\dots,(v_{s_{k_j}},k_j-1),(v_{s_{k_j+1}},k_j+1),\dots,(v_p,q)]\\
    &=(1\otimes d^{k_j})(\sigma\odot\nu_{k_j}^{-1}(s))
  \end{align*}
  where $\lambda_k^{-1}(s)$ and $\nu_k^{-1}(s)$ make sense because $\lambda_k$ and $\nu_k$ are injective. Thus by \eqref{sign}, we get
  \begin{align*}
    &\sum_{s\in\mathcal{S}(p,q)}\left(\sum_{j\in s(\rightarrow)}(-1)^j\sgn(s)H_*(\partial_j(\sigma\odot s))+\sum_{j\in s(\uparrow)}(-1)^j\sgn(s)H_*(\partial_j(\sigma\odot s))\right)\\
    &=\sum_{s\in\mathcal{S}(p,q)}\left(\sum_{j\in s(\rightarrow)}(-1)^j\sgn(s)H_*((\partial_{j-k_j}\sigma)\odot\lambda_{k_j}^{-1}(s))\right.\\
    &\quad\left.+\sum_{j\in s(\uparrow)}(-1)^j\sgn(s)H_*((1\otimes d^{k_j})(\sigma\odot\nu_{k_j}^{-1}(s)))\right)\\
    &=(-1)^q\sum_{i=0}^p\sum_{s\in\mathcal{S}(p-1,q)}(-1)^i\sgn(s)H_*((\partial_i\sigma)\odot s)\\
    &\quad+\sum_{i=0}^q\sum_{s\in\mathcal{S}(p,q-1)}(-1)^i\sgn(s)H_*\circ(1\otimes d^i)(\sigma\odot s)\\
    &=(-1)^qP_H\partial(\sigma)+\sum_{i=0}^q(-1)^iP_{H\circ(1\otimes d^i)}(\sigma).
  \end{align*}
  For $j\in s(\rotatebox[origin=c]{-90}{$\Lsh$})$, we have
  \[
    \partial_j(\sigma\odot s)=\partial_j(\sigma\odot\alpha_j(s))\quad\text{and}\quad\sgn(s)=-\sgn(\alpha_j(s)).
  \]
  Then by Lemma \ref{flip}, we get
  \begin{align*}
    &\sum_{s\in\mathcal{S}(p,q)}\sum_{j\in s(\rotatebox[origin=c]{-90}{$\Lsh$})}(-1)^j\sgn(s)H_*(\partial_j(\sigma\odot s))\\
    &=-\sum_{s\in\mathcal{S}(p,q)}\sum_{j\in s(\rotatebox[origin=c]{-90}{$\Lsh$})}(-1)^j\sgn(\alpha_j(s))H_*(\partial_j(\sigma\odot\alpha_j(s)))\\
    &=-\sum_{s\in\mathcal{S}(p,q)}\sum_{j\in s(\Rsh)}(-1)^j\sgn(s)H_*(\partial_j(\sigma\odot s))
  \end{align*}
  implying
  \[
    \sum_{s\in\mathcal{S}(p,q)}\left(\sum_{j\in s(\Rsh)}(-1)^j\sgn(s)H_*(\partial_j(\sigma\odot s))+\sum_{j\in s(\rotatebox[origin=c]{-90}{$\Lsh$})}(-1)^j\sgn(s)H_*(\partial_j(\sigma\odot s))\right)=0.
  \]
  Thus we obtain
  \[
    \partial P_H(\sigma)=(-1)^qP_H\partial(\sigma)+\sum_{i=0}^q(-1)^iP_{H\circ(1\otimes d^i)}(\sigma)
  \]
  which is the desired identity.
\end{proof}

We define a higher homotopy between simplicial complexes that we are going to use. Let $K$ be an ordered simplicial complex, and let $I$ denote a partition $V(K)=I_0\sqcup I_1\sqcup\cdots\sqcup I_q$ such that $I_0,\ldots,I_q$ are non-empty and pairwise disjoint. For $i=0,1,\ldots,q$, we define a simplicial map $h_I^i\colon K\to K^{\star(q+1)}$ by
\[
  h_I^i(\sigma)=(\sigma\cap I_0)\sqcup\cdots\sqcup(\sigma\cap I_{i-1})\sqcup(\sigma\cap(I_i\sqcup\cdots\sqcup I_q))\sqcup\underbrace{\emptyset\sqcup\cdots\sqcup\emptyset}_{q-i}
\]
for a simplex $\sigma$ of $K$. Then we have:

\begin{lemma}
  \label{h_I}
  For $i=0,1,\ldots,q-1$, the image of $h_I^i$ is in $K^{\star q}\subset K^{\star q}\star K=K^{\star(q+1)}$.
\end{lemma}

By definition, the map $h_I^q$ is concerned with the $(q+1)$-fold products in $\mathcal{C}^*(K)$, and for $i=0,1,\ldots,q-1$, we may think of $h_I^i$ as a deformation of $h_I^q$ into the first $(i+1)$-th join $K^{\star(i+1)}\subset K^{\star(i+1)}\star K^{\star(q-i)}=K^{\star(q+1)}$. We define a higher homotopy among these deformations.

\begin{lemma}
  The map
  \[
    V(K)\times\{0,1,\ldots,q\}\to\underbrace{V(K)\sqcup\cdots\sqcup V(K)}_{q+1},\quad(v,i)\mapsto h_I^i(v).
  \]
  defines a simplicial map $H_I\colon K\otimes\Delta^q\to K^{\star(q+1)}$.
\end{lemma}

\begin{proof}
  Let $\sigma=\{(v_0,w_0),\ldots,(v_p,w_p)\}$ be a simplex of $K\otimes\Delta^q$. Then $\{v_0,\ldots,v_p\}$ is a simplex of $K$, hence its subsets are simplices of $K$. By definition, we have
  \[
    H_I(\sigma)=\sigma_0\sqcup\cdots\sqcup\sigma_q.
  \]
  where each $\sigma_i$ is a subset of $\{v_0,\ldots,v_p\}$. Then each $\sigma_i$ is a simplex of $K$, and so $H_I(\sigma)$ is a simplex of $K^{\star(q+1)}$, completing the proof.
\end{proof}

We denote the higher prism operator $P_{H_I}$ by $P(I)$. As mentioned above, $P(I)$ is the identity map of $K$ for $q=0$. For $i=0,1,\ldots q-1$, let $I(i)$ be the partition $V(K)=J_0\sqcup\cdots\sqcup J_{q-1}$ such that
\[
  J_k=
  \begin{cases}
    I_k&0\le k\le i-1\\
    I_{i}\sqcup I_{i+1}&k=i\\
    I_{k+1}&i+1\le k\le q-1.
  \end{cases}
\]
For $i=0,1,\ldots,q-1$, we set
\[
  \mu_i=(j_{I_0\sqcup\cdots\sqcup I_{i}}\star j_{I_{i+1}\sqcup\cdots\sqcup I_{q}})\circ\iota_{I_0\sqcup\cdots\sqcup I_{i},I_{i+1}\sqcup\cdots\sqcup I_{q}}\colon K\to K\star K.
\]
By definition, the partition $I$ defines $\mu_i$, and we write $\mu_i^I$ when we specify the defining partition $I$. We specialize Lemma \ref{boundary condition} to the higher prism operators $P(I)$.

\begin{lemma}
  \label{boundary condition I}
  There is an identity
  \begin{multline*}
    \partial P(I)+(-1)^{q-1}P(I)\partial\\
    =\sum_{i=0}^{q-1}(-1)^i(1_{K^{\star i}}\star\mu_i\star 1_{K^{\star(q-i-1)}})_*\circ P(I(i))+(-1)^qP_{H_I\circ(1_K\otimes d^q)}.
  \end{multline*}
\end{lemma}

\begin{proof}
  Let $\sigma$ be a simplex of $K$. For $0\le j<i<q$, we have
  \begin{align*}
    &(1_{K^{\star i}}\star\mu_i\star 1_{K^{\star(q-i-1)}})\circ h_{I(i)}^j(\sigma)\\
    &=(1_{K^{\star i}}\star\mu_i\star 1_{K^{\star(q-i-1)}})((\sigma\cap I_0)\sqcup\cdots\sqcup(\sigma\cap I_{j-1})\\
    &\quad\sqcup(\sigma\cap(I_{j}\sqcup\cdots\sqcup I_{q}))\sqcup\underbrace{\emptyset\sqcup\cdots\sqcup\emptyset}_{q-j-1})\\
    &=(\sigma\cap I_0)\sqcup\cdots\sqcup(\sigma\cap I_{j-1})\sqcup(\sigma\cap(I_{j}\sqcup\cdots\sqcup I_{q}))\sqcup\underbrace{\emptyset\sqcup\cdots\sqcup\emptyset}_{q-j}\\
    &=h_I^j(\sigma).
  \end{align*}
  For $0\le i=j<q$, we have
  \begin{align*}
    &(1_{K^{\star i}}\star\mu_i\star 1_{K^{\star(q-i-1)}})\circ h_{I(i)}^i(\sigma)\\
    &=(1_{K^{\star i}}\star\mu_i\star 1_{K^{\star(q-i-1)}})((\sigma\cap I_0)\sqcup\cdots\sqcup(\sigma\cap I_{i-1})\\
    &\quad\sqcup(\sigma\cap(I_{i}\sqcup\cdots\sqcup I_{q}))\sqcup\underbrace{\emptyset\sqcup\cdots\sqcup\emptyset}_{q-i-1})\\
    &=(\sigma\cap I_0)\sqcup\cdots\sqcup(\sigma\cap I_{i})\sqcup(\sigma\cap(I_{i+1}\sqcup\cdots\sqcup I_{q}))\sqcup\underbrace{\emptyset\sqcup\cdots\sqcup\emptyset}_{q-i-1}\\
    &=h_I^{i+1}(\sigma)=h_I^{j+1}(\sigma),
  \end{align*}
  and for $0\le i<j<q$, we have
  \begin{align*}
    &(1_{K^{\star i}}\star\mu_i\star 1_{K^{\star(q-i-1)}})\circ h_{I(i)}^j(\sigma)\\
    &=(1_{K^{\star i}}\star\mu_i\star 1_{K^{\star(q-i-1)}})((\sigma\cap I_0)\sqcup\cdots\sqcup(\sigma\cap I_{i-1})\sqcup(\sigma\cap(I_{i}\sqcup I_{i+1}))\\
    &\quad\sqcup(\sigma\cap I_{i+2})\sqcup\cdots\sqcup(\sigma\cap I_{j})\sqcup(\sigma\cap(I_{j+1}\sqcup\cdots\sqcup I_{q}))\sqcup\underbrace{\emptyset\sqcup\cdots\sqcup\emptyset}_{q-j-1})\\
    &=(\sigma\cap I_1)\sqcup\cdots\sqcup(\sigma\cap I_{j})\sqcup(\sigma\cap(I_{j+1}\sqcup\cdots\sqcup I_{q}))\sqcup\underbrace{\emptyset\sqcup\cdots\sqcup\emptyset}_{q-j-1}\\
    &=h_I^{j+1}(\sigma).
  \end{align*}
  Then the statement follows from Lemma \ref{boundary condition}.
\end{proof}

\begin{lemma}
  \label{prism}
  Let $a_i\in C^*(K)$ for $i=0,1,\ldots,q$. Then we have
  \begin{multline*}
    (P(I)^*\delta+(-1)^{q-1}\delta P(I)^*)(a_0\star\cdots\star a_q)\\
    =\sum_{i=0}^{q-1}(-1)^iP(I(i))^*(1_{K^{\star i}}\star\mu_i\star 1_{K^{\star(q-i-1)}})^*(a_0\star\cdots\star a_q).
  \end{multline*}
\end{lemma}

\begin{proof}
  By Lemma \ref{h_I} and the fact that the inclusion $K^{\star q}\to K^{\star q}\star K=K^{\star(q+1)}$ is trivial in cohomology, we have $(H_I\circ(1\otimes d^q))^*(a_0\star\cdots\star a_q)=0$, implying $P_{H_I\circ(1\otimes d^q)}^*(a_0\star\cdots\star a_q)=0$. Then by Lemma \ref{boundary condition I}, the proof is finished.
\end{proof}


\section{The dga $\mathcal{C}^*(K)$}\label{the complex}

Let $K$ be an (ordered) simplicial complex. In this section, we introduce the dga $\mathcal{C}^*(K)$ which is homotopy equivalent to the Koszul complex of the Stanley-Reisner ring $\mathbb{F}[K]$, and show its properties concerning (higher) Massey products.

We define
\[
  \mathcal{C}^p(K)=\bigoplus_{\emptyset\ne I\subset V(K)}\widetilde{C}^{p-|I|-1}(K_I)
\]
for $p>0$ and $\mathcal{C}^0(K)=\mathbb{F}$, where $\widetilde{C}^*(K)$ denotes the reduced cochain complex of $K$. The differential of $\mathcal{C}^*(K)$ is defined by the coboundary maps $\delta$ of $\widetilde{C}^*(K_I)$, and so we have
\[
  H^p(\mathcal{C}^*(K))=\bigoplus_{\emptyset\ne I\subset V(K)}\widetilde{H}^{p-|I|-1}(K_I)
\]
for $p>0$ and $H^0(\mathcal{C}^*(K))=\mathbb{F}$. The product of $\mathcal{C}^*(K)$ is defined by setting
\[
  \widetilde{C}^{p-|I|-1}(K_I)\otimes\widetilde{C}^{q-|J|-1}(K_J)\to\widetilde{C}^{p+q-|I\cup J|-1}(K_{I\cup J})
\]
to be trivial for $I\cap J\ne\emptyset$ and to be the composite
\[
  \widetilde{C}^{p-|I|-1}(K_I)\otimes\widetilde{C}^{q-|J|-1}(K_J)\xrightarrow{\star}\widetilde{C}^{q-|I\cup J|-1}(K_I\star K_J)\xrightarrow{\iota_{I,J}^*}\widetilde{C}^{p+q-|I\cup J|-1}(K_{I\cup J})
\]
for $I\cap J=\emptyset$.

Recall that the Stanley-Reisner ring of a simplicial complex $K$ with vertex set $\{1,2,\ldots,n\}$ is defined by
\[
  \mathbb{F}[K]=\mathbb{F}[v_1,\ldots,v_n]/(v_{i_1}\cdots v_{i_k}\mid\{i_1,\ldots,i_k\}\not\in K)
\]
where we set $|v_i|=2$. Then the Koszul complex of $\mathbb{F}[K]$ (with respect to $v_1,\ldots,v_n$) is a dga defined by
\[
  \mathcal{K}(K)=\mathbb{F}[K]\otimes\Lambda(x_1,\ldots,x_n),\quad dv_i=0,\quad dx_i=v_i
\]
where $|x_i|=1$. Let
\[
  \mathcal{R}(K)=\mathcal{K}(K)/(v_i^2,\,v_ix_i\mid i=1,2,\ldots,n).
\]
Then $\mathcal{R}(K)$ is a dga. Baskakov, Buchstaber and Panov \cite{BBP} studied the cellular cochain complex of a space called the moment-angle complex for $K$, which is actually $\mathcal{R}(K)$, and proved that the projection $\mathcal{K}(K)\to\mathcal{R}(K)$ is a homotopy equivalence. On the other hand, we can define a map
\[
  \phi\colon\mathcal{C}^*(K)\to\mathcal{R}(K)
\]
by $\phi(\sigma^*)=v_{i_1}\cdots v_{i_k}x_{j_1}\cdots x_{j_l}$ for $\sigma^*\in\widetilde{C}^*(K_I)\subset\mathcal{C}^*(K)$, where $\sigma^*$ is the Kronecker dual of a simplex $\sigma=\{i_1<\cdots<i_k\}\in K_I$ and $I-\sigma=\{j_1<\cdots<j_l\}$. It is easy to see that the map $\phi$ is an isomorphism of dga's, and so we get:

\begin{proposition}
  \label{C(K)-K(K)}
  There is a homotopy equivalence
  \[
    \mathcal{C}^*(K)\simeq\mathcal{K}(K).
  \]
\end{proposition}

We recall the definition of (higher) Massey products. We refer to \cite{K,M} for details, where we follow the sign convention of May \cite{M}. Let $A$ be a dga. We set
\[
  \overline{a}=(-1)^{|a|+1}a
\]
for $a\in A$, where $|a|$ denotes the degree of $a$. Suppose we are given cohomology classes $\alpha_0,\ldots,\alpha_q\in H^*(A)$ for $q\ge 1$. We say that a set $\{a_{i,j}\in A\mid 0\le i\le j\le q,\,(i,j)\ne(0,q)\}$ is a defining system for $\alpha_0,\ldots,\alpha_q$ if $a_{i,i}$ represents $\alpha_i$ for $i=0,1,\ldots,q$ and the identity
\[
  da_{i,j}=\sum_{k=i}^{j-1}\overline{a}_{i,k}a_{k+1,j}
\]
holds for $i<j$. If there is a defining system for $\alpha_0,\ldots,\alpha_q$, then we say that the Massey product $\langle \alpha_0,\ldots,\alpha_q\rangle$ is defined, which is the set of all cohomology classes represented by a cocycle $\sum_{i=0}^{q-1}\overline{a}_{0,i}a_{i+1,q}$, where $\{a_{i,j}\}$ ranges over all defining systems for $\alpha_0,\ldots,\alpha_q$. We say that the Massey product of $\alpha_0,\ldots,\alpha_q$ is trivial if
\[
  0\in\langle \alpha_0,\ldots,\alpha_q\rangle.
\]
Observe that for $q=1$, the Massey product of $\alpha_0,\alpha_1$ is always defined, and $\langle\alpha_0,\alpha_1\rangle=\{\alpha_0\alpha_1\}$. Then the triviality of all Massey products implies the triviality of all products, and so we will not mention the triviality of products when we consider the triviality of all Massey products.

Observe that Massey products are preserved by a homotopy equivalence. Then by Proposition \ref{C(K)-K(K)}, we get:

\begin{proposition}
  \label{C(K)}
  A simplicial complex $K$ is $\mathbb{F}$-Golod if and only if all Massey products in $H^*(\C^*(K))$ are trivial.
\end{proposition}

Thus we will consider Massey products in $H^*(\C^*(K))$ for Golodness of a simplicial complex $K$, instead of $H^*(\mathcal{K}(K))$.

\begin{lemma}
  \label{non-disjoint}
  Let $\alpha_i\in\widetilde{H}^*(K_{I_i})\subset H^*(\mathcal{C}^*(K))$ for $i=0,1,\ldots,q$ with $q\ge 1$. If $I_i\cap I_j\ne\emptyset$ for some $i<j$ and the Massey product of $\alpha_0,\ldots,\alpha_q$ in $H^*(\mathcal{C}^*(K))$ is defined, then $\langle \alpha_0,\ldots,\alpha_q\rangle$ is trivial.
\end{lemma}

\begin{proof}
  Take any defining system $\{a_{k,l}\}$ for $\alpha_0,\ldots,\alpha_q$. For $0\le k\le l\le q$ with $(k,l)\ne(0,q)$, we define
  \[
    b_{k,l}=
    \begin{cases}
      0& k\leq i<j\leq l,\\
      a_{k,l} &\text{otherwise}.
    \end{cases}
  \]
  Since $b_{k,k}=a_{k,k}$ for $k=0,1,\ldots,q$, each $b_{k,k}$ represents $\alpha_k$. If $(k,l)$ does not satisfy $k\le i<j\le l$, then we have
  \[
    \delta b_{k,l}=\delta a_{k,l}=\sum_{p=k}^{l-1}\overline{a}_{k,p}a_{p+1,l}=\sum_{p=k}^{l-1}\overline{b}_{k,p}b_{p+1,l}.
  \]
  Suppose $k\leq i<j\leq l$. If $p<i$, then $p+1\leq i<j\leq l$, implying $b_{p+1,l}=0$. If $p\ge j$, then $k\le i<j\le p$, implying $b_{k,p}=0$. Moreover, if $i\leq p<j$, then $\overline{b}_{k,p}b_{p+1,l}=0$ because $\overline{b}_{k,p}\in \widetilde{C}^*(K_{I_k\sqcup\cdots\sqcup I_p})$ and $b_{p+1,l}\in \widetilde{C}^*(K_{I_{p+1}\sqcup\cdots\sqcup I_l})$ such that $(I_k\sqcup\cdots\sqcup I_p)\cap(I_{p+1}\sqcup\cdots\sqcup I_l)\neq\emptyset$. Hence we get
  \[
    \delta b_{k,l}=0=\sum_{p=k}^{l-1}\overline{b}_{k,p}b_{p+1,l}
  \]
  and so $\{b_{k,j}\mid 0\le k\le l\le q,\,(k,l)\ne(0,q)\}$ is a defining system for $\alpha_0,\ldots,\alpha_q$ such that
  \[
    \sum_{p=0}^{q-1}\overline{b}_{0,p}b_{p+1,q}=0.
  \]
  Thus the Massey product of $\alpha_0,\ldots,\alpha_q$ is trivial, completing the proof.
\end{proof}

Let $\alpha_i\in\widetilde{H}^*(K_{I_i})$ for $i=0,1,\ldots,q$ with $q\ge 1$. If the Massey product of $\alpha_0,\ldots,\alpha_q$ in $H^*(\mathcal{C}^*(K))$ is defined, then we say that its support is $I_0\cup\cdots\cup I_q$.

\begin{lemma}
  \label{full subcomplex}
  Let $\mathcal{C}$ be a class of simplicial complexes which is closed under taking full subcomplexes. If for each $K\in\mathcal{C}$, all Massey products in $H^*(\mathcal{C}^*(K))$ with support $V(K)$ are trivial, then all simplicial complexes in $\mathcal{C}$ are $\mathbb{F}$-Golod.
\end{lemma}

\begin{proof}
  By definition, $\mathcal{C}^*(K)$ is natural with respect to subcomplexes of $K$, that is, an inclusion $i\colon L\to K$ of a subcomplex induces a dga homomorphism $i^*\colon \mathcal{C}^*(K)\to\mathcal{C}^*(L)$. Then for any given Massey product $\langle\alpha_0,\ldots,\alpha_q\rangle$ in $H^*(\mathcal{C}^*(K))$ with support $\emptyset\ne I\subset V(K)$, we have
  \[
    j_I^*(\langle\alpha_0,\ldots,\alpha_q\rangle)\subset\langle j_I^*(\alpha_0),\ldots,j_I^*(\alpha_q)\rangle\subset H^*(\mathcal{C}^*(K_I)).
  \]
  Suppose that $K$ belongs to a class $\mathcal{C}$ in the statement. Since the Massey product $\langle j_I^*(\alpha_0),\ldots,j_I^*(\alpha_q)\rangle$ in $H^*(\mathcal{C}^*(K_I))$ has support $I=V(K_I)$, it is trivial by assumption. Observe that $\mathcal{C}^*(K_I)$ is a retract of $\mathcal{C}^*(K)$. Then we obtain that the Massey product $\langle\alpha_0,\ldots,\alpha_q\rangle$ itself is trivial, completing the proof.
\end{proof}

\begin{lemma}
  \label{triviality}
  Let $\alpha_i\in\widetilde{H}^*(K_{I_i})$ for $i=0,1,\ldots,q$ with $q\ge 1$ such that $V(K)=I_0\sqcup\cdots\sqcup I_q$. Suppose that there are $a_{i,j}\in\widetilde{C}^*(K)$ for $0\le i\le j\le q$\ such that $j_{I_i}^*(a_{i,i})$ represents $\alpha_i$ for $i=0,1,\ldots,q$ and
  \begin{equation}
    \label{condiotion}
    \delta a_{i,j}=\sum_{k=i}^{j-1}\mu_{k}^*(\overline{a}_{i,k}\star a_{k+1,j})
  \end{equation}
  for $0\le i<j\le q$. Then the Massey product of $\alpha_0,\ldots,\alpha_q$ is defined and trivial.
\end{lemma}

\begin{proof}
  For $0\le i\le j\le q$, let $b_{i,j}=j^*_{I_i\sqcup\cdots\sqcup I_j}(a_{i,j})$. Then for $0\le i\le k<j\le q$,
  \begin{align*}
    &j_{I_i\sqcup\cdots\sqcup I_j}^*\circ\mu_k^*(\overline{a}_{i,k}\star a_{k+1,j})\\
    &=((j_{I_0\sqcup\cdots\sqcup I_{k}}\star j_{I_{k+1}\sqcup\cdots\sqcup I_q})\circ\iota_{I_0\sqcup\cdots\sqcup I_{k},I_{k+1}\sqcup\cdots\sqcup I_q}\circ j_{I_i\sqcup\cdots\sqcup I_j})^*(\overline{a}_{i,k}\star a_{k+1,j})\\
    &=((j_{I_i\sqcup\cdots\sqcup I_k}\star j_{I_{k+1}\sqcup\cdots\sqcup I_j})\circ\iota_{I_i\sqcup\cdots\sqcup I_{k},I_{k+1}\sqcup\cdots\sqcup I_j})^*(\overline{a}_{i,k}\star a_{k+1,j})\\
    &=(\iota_{I_i\sqcup\cdots\sqcup I_{k},I_{k+1}\sqcup\cdots\sqcup I_j})^*(\overline{b}_{i,k}\star b_{k+1,j})\\
    &=\overline{b}_{i,k}b_{k+1,j}
  \end{align*}
  where the product in the last line is taken in $\mathcal{C}^*(K)$. Then for $0\le i<j\le q$, we get an identity in $\mathcal{C}^*(K)$
  \begin{equation}
    \label{b}
    \delta b_{i,j}=\sum_{k=i}^{j-1}\overline{b}_{i,k}b_{k+1,j}.
  \end{equation}
  By assumption, $b_{i,i}$ represents $\alpha_i$ for each $i$, and so by \eqref{b}, $\{b_{i,j}\mid 0\le i\le j\le q,\,(i,j)\ne(0,q)\}$ is a defining system for $\alpha_0,\ldots,\alpha_q$ in $H^*(\mathcal{C}^*(K))$. Then the Massey product $\langle\alpha_0,\ldots,\alpha_q\rangle$ in $H^*(\mathcal{C}^*(K))$ is defined, so the RHS of \eqref{b} for $(i,j)=(0,q)$ represents an element of $\langle\alpha_0,\ldots,\alpha_q\rangle$. Thus \eqref{b} for $(i,j)=(0,q)$ shows that $\langle\alpha_0,\ldots,\alpha_q\rangle$ is trivial, completing the proof.
\end{proof}


\section{Proof of Theorem \ref{main}}\label{proof of theorem}

In this section, we prove Theorem \ref{main}. Let $K$ be an (ordered) simplicial complex. To prove Theorem \ref{main}, we need to show that all Massey products in $H^*(\mathcal{C}^*(K))$ are trivial whenever $K$ is $\mathbb{F}$-tight. We consider the special case that the support of Massey products are $V(K)$. Take any $\alpha_i\in\widetilde{H}^*(K_{I_i})\subset H^*(\mathcal{C}^*(K))$ with $\emptyset\ne I_i\subset V(K)$ for $i=0,1,\ldots,q$ such that $V(K)=I_0\cup\cdots\cup I_q$. We shall consider the Massey product of $\alpha_0,\ldots,\alpha_q$. By Lemma \ref{non-disjoint}, if $I_0,\ldots,I_q$ are not pairwise disjoint, then the Massey product of $\alpha_0,\ldots,\alpha_q$ is trivial whenever it is defined. So we may suppose that $I_0,\ldots,I_q$ are pairwise disjoint, and by Lemma \ref{triviality}, it is sufficient to construct $a_{ij}\in\widetilde{C}^*(K)$ for $0\le i\le j\le q$ satisfying the condition of Lemma \ref{triviality}.

\begin{lemma}
  \label{i-j=0}
  If $K$ is $\mathbb{F}$-tight, then for each $i=0,1,\ldots,q$, there is a cocycle $a_{i,i}\in\widetilde{C}^*(K)$ such that $j_{I_i}^*(a_{i,i})$ represents $\alpha_i$.
\end{lemma}

\begin{proof}
  Since $K$ is $\mathbb{F}$-tight, the map $j_{I_i}^*\colon H^*(K)\to H^*(K_{I_i})$ is surjective for each $i$. Then the statement follows.
\end{proof}

We prove the existence of $a_{i,j}$ for $0\le i<j\le q$ by induction on $j-i$.

\begin{lemma}
  \label{j-i=1}
  If there are cocycles $a_{i,i}\in\widetilde{C}^*(K)$ such that each $j_{I_i}^*(a_{i,i})$ represents $\alpha_i$ for $i=0,1,\ldots,q$, then for $i=0,1,\ldots,q-1$, there are $a_{i,i+1}\in\widetilde{C}^*(K)$ satisfying
  \[
    \delta a_{i,i+1}=\mu_i^*(\overline{a}_{i,i}\star a_{i+1,i+1}).
  \]
\end{lemma}

\begin{proof}
  Let $J$ denote the partition $V(K)=J_0\sqcup J_1$ such that $J_0=I_0\sqcup\cdots\sqcup I_i$ and $J_1=I_{i+1}\sqcup\cdots\sqcup I_q$. By Lemma \ref{prism}, we have
  \begin{align*}
    \mu_i^*(\overline{a}_{i,i}\star a_{i+1,i+1})&=P(J(0))^*(\mu_i^*(\overline{a}_{i,i}\star a_{i+1,i+1}))\\
    &=(P(J)^*\delta+\delta P(J)^*)(\overline{a}_{i,i}\star a_{i+1,i+1})\\
    &=\delta P(J)^*(\overline{a}_{i,i}\star a_{i+1,i+1})
  \end{align*}
  where $P(J(0))$ is the identity map of $K$ and $P(J)^*\delta(\overline{a}_{i,i}\star a_{i+1,i+1})=0$ as $a_{i,i}$ and $a_{i+1,i+1}$ are cocycles. Then $a_{i,i+1}=P(J)^*(\overline{a}_{i,i}\star a_{i+1,i+1})$ is the desired element.
\end{proof}

Now we assume that for $j-i<n$, we have the desired $a_{i,j}$, and prove the $j-i=n$ case. Let $s=(s_0,\ldots,s_{k+1})\in\mathcal{S}(j-i,k)$. We define the partition $I^s$ by $V(K)=J_0\sqcup\cdots\sqcup J_k$ such that
\[
  J_p=
  \begin{cases}
    I_{i_0}\sqcup\cdots\sqcup I_{i_1}&p=0\\
    I_{i_p+1}\sqcup\cdots\sqcup I_{i_{p+1}}&p=1,\ldots,k
  \end{cases}
\]
where $i_p=i+s_p$ for $p=0,1,\ldots,k+1$. We set
\[
  \epsilon(s)=
  \begin{cases}
    (|a_{i_1+1,i_2}|+1)+(|a_{i_3+1,i_4}|+1)+\cdots+(|a_{i_{k-1}+1,i_k}|+1)&(k\text{ is even})\\
    (|a_{i_0,i_1}|+1)+(|a_{i_2+1,i_3}|+1)+\cdots+(|a_{i_{k-1}+1,i_k}|+1)&(k\text{ is odd})
  \end{cases}
\]
for $k\ge 1$, and
\[
  a^s=a_{i_0,i_1}\star a_{i_1+1,i_2}\star\cdots\star a_{i_k+1,i_{k+1}}.
\]
We also set
\[
  \epsilon(k)=
  \begin{cases}
    0&(k\equiv 1,2\mod 4)\\
    1&(k\equiv 0,3\mod 4).
  \end{cases}
\]
Let $\widehat{\mathcal{S}}(p,q)$ denote the subset of $\mathcal{S}(p,q)$ consisting of $(p,q)$-shuffles $(s_0,\ldots,s_{q+1})$ satisfying $0=s_0\le s_1<s_2<\cdots<s_{q+1}=p$. By definition, $\widehat{\mathcal{S}}(p,q)=\emptyset$ for $q>p$.

\begin{lemma}
  \label{decomposition}
  If $a_{i,j}$ exists for $j-i<n$, then for $j-i=n$ and $k=1,2,\ldots,j-i$,
  \begin{equation}
    \label{partition}
    \sum_{p=i}^{j-1}\mu_p^*(\overline{a}_{i,p}\star a_{p+1,j})\equiv\sum_{s\in\widehat{\mathcal{S}}(j-i,k)}(-1)^{\epsilon(s)+\epsilon(k)}P(I^s)^*\delta(a^s)\mod\mathrm{Im}\,\delta.
  \end{equation}
\end{lemma}

\begin{proof}
  We induct on $k$. By Lemma \ref{prism}, we have
  \begin{alignat*}{3}
    \mu_p^*(\overline{a}_{i,p}\star a_{p+1,j})&=P(I^s(0))^*(\mu_p^*(\overline{a}_{i,p}\star a_{p+1,j}))\\
    &=(P(I^s)^*\delta+\delta P(I^s)^*)(\overline{a}_{i,p}\star a_{p+1,j})\\
    &\equiv(-1)^{\epsilon(s)+\epsilon(1)}P(I^s)^*\delta(a^s)&&\mod\mathrm{Im}\,\delta
  \end{alignat*}
  where $s=(0,p-i,j-i)\in\widehat{\mathcal{S}}(j-i,1)$ and $P(I^s(0))$ is the identity map of $K$. Then since $\widehat{\mathcal{S}}(j-i,1)=\{(0,p-i,j-i)\in\mathcal{S}(j-i,1)\mid i\le p<j\}$, the $k=1$ case is proved. We suppose that \eqref{partition} holds for $k$, and consider the $k+1$ case. Let $s=(s_0,\ldots,s_{k+1})\in\widehat{\mathcal{S}}(j-i,k)$. We can define
  \[
    b_0=a_{i_0,i_1}\star\overline{a}_{i_1+1,i_2}\quad\text{and}\quad b_p=a_{i_p+1,i_{p+1}}\star\overline{a}_{i_{p+1}+1,i_{p+2}}
  \]
  for $p$ even. Clearly, we have
  \[
    \overline{b}_0=\overline{a}_{i_0,i_1}\star a_{i_1+1,i_2}\quad\text{and}\quad \overline{b}_p=\overline{a}_{i_p+1,i_{p+1}}\star a_{i_{p+1}+1,i_{p+2}}
  \]
  for $p$ even. Then for $k$ even, we have
  \begin{align*}
    (-1)^{\epsilon(s)}\delta(a^s)&=\delta(b_0\star b_2\star\cdots\star b_{k-2}\star a_{i_k+1,i_{k+1}})\\
    &=\sum_{p=0}^{\frac{k-2}{2}}\overline{b}_0\star\cdots\star\overline{b}_{2p-2}\star(\delta b_{2p})\star b_{2p+2}\star\cdots\star b_{k-2}\star a_{i_k+1,i_{k+1}}\\
    &\quad+\overline{b}_0\star\cdots\star\overline{b}_{k-2}\star(\delta a_{i_k+1,i_{k+1}}).
  \end{align*}
  For $p=0,1,\ldots,k$, let $\mathcal{S}_p(s)$ denote the set of all $(j-i,k+1)$-shuffles of the form $(s_0,\ldots,s_p,t,s_{p+1},\ldots,s_{k+1})$ satisfying $s_p<t<s_{p+1}$. Then we have
  \begin{align*}
    &(\delta b_0)\star b_2\star\cdots\star b_{k-2}\star a_{i_k+1,i_{k+1}}\\
    &=(\delta a_{i_0,i_1})\star\overline{a}_{i_1+1,i_2}\star b_2\star\cdots\star b_{k-2}\star a_{i_k+1,i_{k+1}}\\
    &\quad-\overline{a}_{i_0,i_1}\star\overline{(\delta a_{i_1+1,i_2})}\star b_2\star\cdots\star b_{k-2}\star a_{i_k+1,i_{k+1}}\\
    &=\sum_{l=i_0}^{i_1-1}\mu_l^*(\overline{a}_{i_0,l}\star a_{l+1,i_1})\star\overline{a}_{i_1+1,i_2}\star b_2\star\cdots\star b_{k-2}\star a_{i_k+1,i_{k+1}}\\
    &\quad-\sum_{l=i_1+1}^{i_2-1}\overline{a}_{i_0,i_1}\star \mu_l^*(a_{i_1+1,l}\star\overline{a}_{l+1,i_2})\star b_2\star\cdots\star b_{k-2}\star a_{i_k+1,i_{k+1}}\\
    &=\sum_{t\in\mathcal{S}_0(s)}(-1)^{\epsilon(t)}(\mu_0^{I^t}\star 1_{K^{\star k}})(a^t)-\sum_{t\in\mathcal{S}_1(s)}(-1)^{\epsilon(t)}(1_K\star\mu_1^{I^t}\star 1_{K^{\star(k-1)}})^*(a^t).
  \end{align*}
  Quite similarly, we have
  \begin{align*}
    &\overline{b}_0\star\cdots\star\overline{b}_{2p-2}\star(\delta b_{2p})\star b_{2p+2}\star\cdots\star b_{k-2}\star a_{i_k+1,i_{k+1}}\\
    &=\sum_{t\in\mathcal{S}_{2p}(s)}(-1)^{\epsilon(t)}(1_{K^{\star(2p)}}\star\mu_{2p}^{I^t}\star 1_{K^{\star(k-2p)}})^*(a^t)\\
    &\quad-\sum_{t\in\mathcal{S}_{2p+1}(s)}(-1)^{\epsilon(t)}(1_{K^{\star(2p+1)}}\star\mu_{2p+1}^{I^t}\star 1_{K^{\star(k-2p-1)}})^*(a^t)
  \end{align*}
  for $p=1,2,\ldots,\frac{k-2}{2}$ and
  \begin{align*}
    &\overline{b}_0\star\cdots\star\overline{b}_{k-2}\star(\delta a_{i_k+1,i_{k+1}})=\sum_{t\in\mathcal{S}_k(s)}(-1)^{\epsilon(t)}(1_{K^{\star k}}\star\mu_k^{I^t})^*(a^t).
  \end{align*}
  Then we get
  \[
    (-1)^{\epsilon(s)}\delta(a^s)=\sum_{p=0}^k(-1)^p\sum_{t\in\mathcal{S}_p(s)}(-1)^{\epsilon(t)}(1_{K^{\star p}}\star\mu_p^{I^{t}}\star 1_{K^{\star(k-p)}})^*(a^{t}).
  \]
  For $k$ odd, we can similarly get
  \begin{align*}
    &(-1)^{\epsilon(s)}\delta(a^s)\\
    &=\delta(\overline{b}_0\star\overline{b}_2\star\cdots\star\overline{b}_{k-1})\\
    &=\sum_{p=0}^{\frac{k-1}{2}}b_0\star\cdots\star b_{2p-2}\star(\delta \overline{b}_{2p})\star\overline{b}_{2p+2}\star\cdots\star\overline{b}_{k-1}\\
    &=-\overline{(\delta a_{i_0,i_1})}\star a_{i_1+1,i_2}\star\overline{b}_2\star\cdots\star\overline{b}_{k-1}+\overline{a}_{i_0,i_1}\star(\delta a_{i_1+1,i_2})\star\overline{b}_2\star\cdots\star\overline{b}_{k-1}\\
    &\quad+\sum_{p=1}^{\frac{k-1}{2}}(-b_0\star\cdots\star b_{2p-2}\star\overline{(\delta a_{i_{2p}+1,i_{2p+1}})}\star a_{i_{2p+1}+1,i_{2p+2}}\star\overline{b}_{2p+2}\star\cdots\star\overline{b}_{k-1}\\
    &\quad+b_0\star\cdots\star b_{2p-2}\star a_{i_{2p}+1,i_{2p+1}}\star(\delta a_{i_{2p+1}+1,i_{2p+2}})\star\overline{b}_{2p+2}\star\cdots\star\overline{b}_{k-1})\\
    &=-\sum_{p=0}^k(-1)^p\sum_{t\in\mathcal{S}_p(s)}(-1)^{\epsilon(t)}(1_{K^{\star p}}\star\mu_p^{I^{t}}\star 1_{K^{\star(k-p)}})^*(a^{t}).
  \end{align*}
  Observe that the map
  \[
    \widehat{\mathcal{S}}(j-i,k+1)\times\{0,1,\ldots,k\}\to\coprod_{s\in\widehat{\mathcal{S}}(j-i,k)}(\mathcal{S}_0(s)\sqcup\mathcal{S}_1(s)\sqcup\cdots\sqcup\mathcal{S}_k(s))
  \]
  sending $(t,p)\in\widehat{\mathcal{S}}(j-i,k+1)\times\{0,1,\ldots,k\}$ to $t\in\mathcal{S}_p(s)$ is bijective, where $s=(t_0,\ldots,t_{p},t_{p+2},\ldots,t_{k+2})\in\widehat{\mathcal{S}}(j-i,k)$ for $t=(t_0,\ldots,t_{k+2})\in\widehat{\mathcal{S}}(j-i,k+1)$. Then we get
  \begin{align*}
    &\sum_{p=i}^{j-1}\mu_p^*(\overline{a}_{i,p}\star a_{p+1,j})\\
    &\equiv\sum_{s\in\widehat{\mathcal{S}}(j-i,k)}(-1)^{\epsilon(s)+\epsilon(k)}P(I^s)^*\delta(a^s)\\
    &\equiv\sum_{s\in\widehat{\mathcal{S}}(j-i,k)}\sum_{p=0}^k(-1)^p\sum_{t\in\mathcal{S}_p(s)}(-1)^{\epsilon(t)+\epsilon(k+1)}P(I^t(p))^*(1_{K^{\star p}}\star\mu_p^{I^t}\star 1_{K^{\star(k-p)}})^*(a^t)\\
    &\equiv\sum_{t\in\widehat{\mathcal{S}}(j-i,k+1)}(-1)^{\epsilon(t)+\epsilon(k+1)}\sum_{p=0}^k(-1)^pP(I^t(p))^*(1_{K^{\star p}}\star\mu_p^{I^t}\star 1_{K^{\star(k-p)}})^*(a^t)
  \end{align*}
  modulo $\mathrm{Im}\,\delta$. By Lemma \ref{prism}, we also get
  \begin{align*}
    &\sum_{t\in\widehat{\mathcal{S}}(j-i,k+1)}(-1)^{\epsilon(t)+\epsilon(k+1)}\sum_{p=0}^k(-1)^pP(I^t(p))^*(1_{K^{\star p}}\star\mu_p^{I^t}\star 1_{K^{\star(k-p)}})^*(a^t)\\
    &\equiv\sum_{t\in\widehat{\mathcal{S}}(j-i,k+1)}(-1)^{\epsilon(t)+\epsilon(k+1)}(P(I^t)^*\delta+(-1)^k\delta P(I^t)^*)(a^t)\\
    &\equiv\sum_{t\in\widehat{\mathcal{S}}(j-i,k+1)}(-1)^{\epsilon(t)+\epsilon(k+1)}P(I^t)^*\delta(a^t)
  \end{align*}
  modulo $\mathrm{Im}\,\delta$. Thus the induction on $k$ proceeds, completing the proof.
\end{proof}

Consider the identity \eqref{partition} for $k=j-i$. Observe that $\widehat{\mathcal{S}}(j-i,j-i)$ consists of a single element $(0,0,1,\ldots,j-i)$, and $a^{(0,0,1,\ldots,j-i)}=a_{i,i}\star\cdots\star a_{j,j}$ which is a cocycle because each $a_{k,k}$ is a cocycle. Then we obtain
\begin{align*}
  \sum_{k=i}^{j-1}\mu_{k}^*(\overline{a}_{i,k}\star a_{k+1,j})&\equiv(-1)^{\epsilon((0,0,1,\ldots,j-i))+\epsilon(j-i+1)}P(I^{(0,0,1,\ldots,j-i)})\delta(a_{i,i}\star\cdots\star a_{j,j})\\
  &\equiv 0
\end{align*}
modulo $\mathrm{Im}\,\delta$. Hence there is $a_{i,j}\in\widetilde{C}^*(K)$ satisfying \eqref{condiotion}, and so the induction on $j-i$ proceeds by Lemmas \ref{j-i=1} and \ref{decomposition}. Thus by Lemma \ref{i-j=0}, we get:

\begin{lemma}
  \label{a}
  If $K$ is $\mathbb{F}$-tight, then there are $a_{ij}\in\widetilde{C}^*(K)$ for $0\le i\le j\le q$ satisfying the condition of Lemma \ref{triviality}.
\end{lemma}

Finally, by Lemma \ref{triviality}, we obtain:

\begin{proposition}
  \label{support V(K)}
  If $K$ is an $\mathbb{F}$-tight complex, then every Massey product in $H^*(\mathcal{C}^*(K))$ with support $V(K)$ is trivial whenever it is defined.
\end{proposition}

Now we are ready to prove Theorem \ref{main}.

\begin{proof}
  [Proof of Theorem \ref{main}]
  Induct on the cardinality $|V(K)|$. For $|V(K)|=1$, then $K$ is obviously $\mathbb{F}$-Golod. We assume that all Massey products in $H^*(\mathcal{C}^*(K))$ are trivial for $|V(K)|<m$, and we consider the $|V(K)|=m$ case. By definition, the class of $\mathbb{F}$-tight complexes are closed under taking full subcomplexes. Then by Lemma \ref{full subcomplex} and the induction hypothesis, it is sufficient to show that all Massey products in $H^*(\mathcal{C}^*(K))$ with support $V(K)$ are trivial. Thus by Proposition \ref{support V(K)}, the proof is finished.
\end{proof}

\end{document}